\documentclass[12pt]{article}
\usepackage[T1]{fontenc}
\usepackage{amssymb}
\usepackage{amsmath}
\usepackage{amsfonts}
\usepackage{amsthm}
\usepackage{comment}
\usepackage{overpic}
\usepackage{color}
\usepackage[showonlyrefs]{mathtools}
\mathtoolsset{showonlyrefs}
\numberwithin{equation}{section}
\newtheorem{theorem}{Theorem}[section]
\newtheorem{lemma}{Lemma}[section]
\newtheorem{proposition}{Proposition}[section]

\theoremstyle{remark}
\newtheorem{remark}{Remark}[section]

\def\B{\mathcal{B}}
\def\R{\mathbb{R}}
\def\Z{\mathbb{Z}}
\def\C{\mathbb{C}}

\def\N{\mathbb{N}}

\def\O{\mathcal{O}}
\def\dist{\operatorname{dist}}
\def\card{\operatorname{card}}
\def\re{\operatorname{Re}}
\def\im{\operatorname{Im}}

\def\diam{\operatorname{diam}}
\def\area{\operatorname{area}}
\def\logarea{\operatorname{logarea}}
\begin{document}
\title{Lebesgue measure of Julia sets and escaping sets of certain entire functions}
\author{Walter Bergweiler}
\date{}
\maketitle
%
%
%
\begin{abstract}
We give criteria for the escaping set and the Julia set of an entire function
to have positive measure. The results are applied to Poincar\'e functions of semihyperbolic 
polynomials and to the Weierstra{\ss} $\sigma$-func\-tion.

\medskip


\medskip

\end{abstract}

\section{Introduction and results} \label{intro}
Let $f$ be a non-linear entire function and let $f^n$ denote the $n$-th iterate of~$f$.
The \emph{Fatou set} $F(f)$ is the set of all $z\in \C$ where the $f^n$ form a normal family;
its complement $J(f)$ is the \emph{Julia set}. The \emph{escaping set} $I(f)$ is the set of all
 $z\in \C$ such that $f^n(z)\to\infty$.
By a result of Eremenko~\cite{Eremenko1989} we have $J(f)=\partial I(f)$.
These sets play a key role in complex dynamics;
see~\cite{Bergweiler1993} and~\cite{Schleicher2010}
for an introduction to the dynamics of transcendental entire functions.

A result of McMullen~\cite[Theorem 1.1]{McMullen1987} says that $J(\sin(\alpha z+\beta))$ has positive
Lebesgue measure for all $\alpha,\beta\in\C$ with $\alpha\neq 0$.
In his proof McMullen actually showed that $I(\sin(\alpha z+\beta))$ has positive measure
and then noted that $I(f)\subset J(f)$ for $f(z)=\sin(\alpha z+\beta)$.
It was later shown by Eremenko and Lyubich~\cite[Theorem~1]{Eremenko1992} that 
$I(f)\subset J(f)$ holds more generally for all transcendental entire functions $f$ for which
the set of critical and asymptotic values is bounded. The class of functions with
the latter property, denoted by~$\B$, is now called the \emph{Eremenko-Lyubich class} and has 
received much attention in transcendental dynamics.

McMullen's result on the measure of $J(\sin(\alpha z+\beta))$ has been extended to various 
classes of functions in~\cite{Aspenberg2012,Bergweiler2016,Sixsmith2015}.
In this paper we give another criterion for the Julia set or escaping set
of an entire function to have positive measure. Perhaps more importantly, we do so by
a method different from those employed in the papers mentioned. Here we only note that
distortion estimates, coming from Koebe's theorem or related results, 
do not occur in the proofs of our main results.

The order $\rho(f)$ of an entire function $f$ is defined by
\begin{equation}\label{4a}
\rho(f)=\limsup_{r\to\infty}\frac{\log\log M(r,f)}{\log r},
\end{equation}
where $M(r,f)=\max_{|z|=r}|f(z)|$ denotes the maximum modulus of~$f$.
The area (i.e., the two-dimensional Lebesgue measure) 
of a measurable subset $A$ of $\C$ is denoted by $\area A$. The \emph{logarithmic area} of $A$ is defined by 
\begin{equation}\label{1e}
\logarea A =\int_A \frac{dx\,dy}{|z|^2} .
\end{equation}
The logarithmic area occurs in transcendental dynamics in~\cite[p.~34]{Bishop2015}
and \cite[p.~575]{Epstein2015}; in the latter paper the term cylindrical area is used.

We are interested in the behavior near $\infty$ and thus 
instead of the logarithmic area of a set $A$ we will usually consider the
logarithmic area of $A\cap\Delta$ where
$\Delta=\{z\colon |z|\geq 1\}$.
\begin{theorem}\label{t1}
Let $f$ be an entire function of finite order. Let $\varepsilon>0$ and suppose that 
\begin{equation}\label{4b}
\logarea\!\left\{ z\in\Delta \colon \left|\frac{zf'(z)}{f(z)}\right| < |z|^{\rho(f)/2+\varepsilon}
\ \text{or} \ 
 \left|f(z)\right| < (1+\varepsilon)|z| \right\} <\infty.
\end{equation}
Then 
\begin{equation}\label{4c}
\logarea(\Delta\backslash I(f))<\infty.
\end{equation}
In particular, $\area I(f)>0$.

If, in addition, $F(f)$ does not have a multiply connected component, then
\begin{equation}\label{4d}
\logarea (\Delta\backslash (I(f)\cap J(f)))<\infty
\end{equation}
and thus $\area (I(f)\cap J(f))> 0$.
\end{theorem}
We consider the example $f(z)=\sin z$. Then $\rho(f)=1$,
\begin{equation}\label{4f1}
|f(z)|\geq \frac12\! \left(e^{|\im z|}-1\right)
\geq 2|z| \quad\text{if }|\im z|\geq \log( 4|z|+1) 
\end{equation}
and
\begin{equation}\label{4f3}
\left|\frac{zf'(z)}{f(z)}\right| =\left|z\cot z\right|
\geq\frac12 |z| \geq |z|^{3/4}
\quad\text{if }|\im z|\geq 1 \text{ and } |z|\geq  16.
\end{equation}
It is easy to see that the set
\begin{equation}\label{4g}
\{ z\in\Delta\colon |\im z|<\log(4|z|+1)\}
\end{equation}
has finite logarithmic area. 
Thus Theorem~\ref{t1} yields that $I(\sin z)$ has positive measure.
Also, a result of Baker~\cite[p.~565]{Baker1984} says that
$F(f)$ does not have multiply connected components
if $f$ is bounded on a curve tending to~$\infty$.
Thus we also find that $J(\sin z)$ has positive measure.

With the same method we could also treat the functions $\sin(\alpha z+\beta)$ 
considered by McMullen and thus obtain another proof of his result that the Julia set
of these functions has positive measure.
More generally, the hypothesis of Theorem~\ref{t1} is satisfied
for example
if $f(z)=P(z)\sin(\alpha z+\beta)$ with a polynomial~$P$.
Moreover, the result of Baker just mentioned holds more generally if
$\log|f(z)|=\O(\log|z|)$ for $z$ on some curve tending to $\infty$; see~\cite[Theorem~10]{Bergweiler1993}.
We thus find that $J(f)$ has positive area for such~$f$.
Note that $f$ is not in the Eremenko-Lyubich class if $P$ is non-constant.

Theorem~\ref{t1} also applies to the functions
\begin{equation} \label{intro2}
f(z)=\sum_{k=0}^{n}a_k\exp\!\left(b_k z\right)
\end{equation}
considered in~\cite{Bergweiler2016,Sixsmith2015}.
Here the $a_k$ and $b_k$ are non-zero constants satisfying
$\arg b_k<\arg b_{k+1} \leq \arg b_k +\pi$ for $0\leq k\leq n-1$ and
$\arg b_0\leq\arg b_{n}-\pi$, with arguments chosen in $[0,2\pi)$.
More generally, one can assume that the $a_k$ are polynomials that do not vanish identically.

A subset $A(f)$ of $I(f)$ called the \emph{fast escaping set} was introduced in~\cite{Bergweiler1999}.
It also plays an important role in transcendental dynamics; see, e.g., \cite{Rippon2005a,Rippon2012}.
In order to define it, let $M^n(r,f)$ denote the $n$-th iterate of $M(r,f)$ with respect
to the first variable; that is,
\begin{equation}\label{4b4}
M^1(r,f)=M(r,f)\quad \text{and} \quad M^{n}(r,f)=M(M^{n-1}(r,f),f) \quad \text{for  } n\geq 2.
\end{equation}
We note that there exists $R>0$ such that $M(r,f)>r$ for $r\geq R$.
With such a value of $R$ the fast escaping set 
$A(f)$ is defined as the set of all $z\in\C$ for which there exists $L\in\N$
such that $|f^n(z)|\geq M^{n-L}(R,f)$ for $n>L$. The definition is independent of the value of~$R$.
\begin{theorem}\label{t2}
Let $f$ be an entire function of finite order.
Let $\varepsilon>0$ and suppose that 
\begin{equation}\label{4b2}
\logarea\!\left\{ z\in\Delta \colon \left|\frac{zf'(z)}{f(z)}\right| < |z|^{\rho(f)/2+\varepsilon}
\ \text{or} \ 
 \left|f(z)\right| < \exp\!\left(|z|^\varepsilon\right) \right\} <\infty.
\end{equation}
Then the conclusion of Theorem~$\ref{t1}$ holds with $I(f)$ replaced by $A(f)$.
\end{theorem}
The arguments used to show that the hypotheses of Theorem~\ref{t1} are satisfied for
$f(z)=\sin z$ or, more generally, for
$f(z)=P(z)\sin(\alpha z+\beta)$ with a polynomial~$P$ and the functions given by~\eqref{intro2},
can easily be modified to show that 
the hypotheses of Theorem~\ref{t2} hold for these functions as well.

We will deduce the above theorems from a more general result which does not involve the order.
To state this result, denote for an entire function $f$ and $a\in\C$
by $n(r,a)$ the number of $a$-points of $f$ in $\{z\colon |z|\leq r\}$. Put
\begin{equation}\label{1a}
n(r)=\max_{a\in\C} n(r,a).
\end{equation}
\begin{theorem}\label{t3}
Let $f$ be a transcendental entire function satisfying
\begin{equation}\label{4b1}
\logarea\!\left\{ z\in\Delta \colon \left|\frac{zf'(z)}{f(z)}\right| 
< n(|z|)^{1/2+\varepsilon}
\ \text{or} \ 
 \left|f(z)\right| < (1+\varepsilon)|z| \right\} <\infty
\end{equation}
for some $\varepsilon>0$.
Then~\eqref{4c} holds.  In particular, $\area I(f)> 0$.

If, in addition, $F(f)$ does not have a multiply connected component, 
then~\eqref{4d} also holds and thus $\area (I(f)\cap J(f))> 0$.
\end{theorem}
In the results above 
the hypotheses concern
both
$|f(z)|$ and $|zf'(z)/f(z)|$.
If $f\in\B$, then $|zf'(z)/f(z)|$ can be bounded in terms of $|f(z)|$.
In fact, we have the following result which follows directly from~\cite[Lemma~1]{Eremenko1992};
see~\cite[Lemma~2]{Bergweiler1995}.
\begin{proposition} \label{la2}
Let $f\in\B$.
Then there exists $R>0$ such that
\begin{equation}\label{pr1}
\left|\frac{zf'(z)}{f(z)}\right| \geq \frac{1}{4\pi} \log \frac{|f(z)|}{R}
\end{equation}
for all $z\in\C$ with $f(z)\neq 0$.
\end{proposition}
The following result is a simple consequence of Theorem~\ref{t2} and Proposition~\ref{la2}.
\begin{theorem}\label{t4}
Let $f\in\B$ be of finite order.
Suppose that
\begin{equation}\label{4b3}
\logarea\!\left\{ z\in\Delta \colon 
\left|f(z)\right| < \exp\!\left(|z|^{\rho(f)/2+\varepsilon}\right) \right\} <\infty.
\end{equation}
for some $\varepsilon>0$.
Then
\begin{equation}\label{4b5}
\logarea (\Delta\backslash A(f))<\infty.
\end{equation}
\end{theorem}
We recall that by the result of Eremenko and Lyubich already mentioned 
we have $A(f)\subset I(f)\subset J(f)$ 
for $f\in\B$. Under the hypotheses of Theorem~\ref{t4} we thus have, in 
particular, 
$\logarea (\Delta\backslash J(f))<\infty$ and hence $\area J(f)>0$.

As an example where Theorems~\ref{t2} and~\ref{t4} apply to we consider certain Poincar\'e functions.
We recall the definition of these functions: let $p$ be a polynomial of degree $d\geq 2$ and 
let $z_0$ be a repelling fixed point of~$p$; that is, $p(z_0)=z_0$ and $\lambda:=p'(z_0)$ satisfies
$|\lambda|>1$. Then Schr\"oder's functional equation 
\begin{equation}\label{4s}
f(\lambda z)=p(f(z))
\end{equation}
has a solution $f$ which is holomorphic in a neighborhood of $0$ and satisfies $f(0)=z_0$;
see~\cite[Theorem 8.2]{Milnor2006}. It can be normalized to satisfy $f'(0)=1$.
This solution $f$ actually extends to a transcendental entire function which
is called the \emph{Poincar\'e function} of $p$ at~$z_0$; see~\cite[Corollary 8.1]{Milnor2006}.
It is well-known \cite[Chapter II, Section~III.8]{Valiron1923} that $\rho(f)=\log d/\log|\lambda|$.
We note that the trigonometric functions arise as Poincar\'e functions of Chebychev polynomials.

The dynamics of Poincar\'e functions have been studied
in~\cite[Section~3]{Epstein2015} 
and~\cite{Mihaljevic2012}.
It is known that $f\in\B$ if and only if the orbit $\{p^n(c)\colon n\in\N\}$ is bounded
for every critical point $c$ of~$p$; see~\cite[Proposition~4.2]{Mihaljevic2012}
or~\cite[Section~3.1]{Epstein2015}.
The latter condition is satisfied if and only if $J(p)$ is connected~\cite[Theorem 9.5]{Milnor2006}.

A polynomial $p$ is called \emph{semihyperbolic} if there exist $\varepsilon>0$ and
$N\in\N$ such that if $z\in J(f)$, $n\in\N$ and $V$ is a component of $p^{-n}(D(z,\varepsilon))$,
then the degree of the proper map $p^n\colon V\to D(z,\varepsilon)$ is at most~$N$.
Here $D(z,\varepsilon)$ denotes the open disk of radius $\varepsilon$ around~$z$.
The concept of semihyperbolicity was introduced by Carleson, Jones and Yoccoz~\cite{Carleson1994},
who gave various characterizations of it.
\begin{theorem}\label{t5}
Let $p$ be a semihyperbolic polynomial without attracting periodic points
 and let $f$ be a Poincar\'e function of~$p$. Then
\begin{equation}\label{4b6}
\logarea (\Delta\backslash (A(f)\cap J(f)))<\infty.
\end{equation}
In particular, $\area(A(f)\cap J(f))>0$.
\end{theorem}
The \emph{filled Julia set} $K(p)$ of a polynomial $p$ is defined by
\begin{equation}\label{1f}
K(p)=\{z\colon p^n(z)\not\to\infty\}.
\end{equation}
We always have $J(p)\subset K(p)$.
Semihyperbolic polynomials have no para\-bolic points and no Siegel disks.
The hypothesis in Theorem~\ref{t5} that $p$ has no attracting periodic points is thus equivalent 
to $J(p)=K(p)$.
The following result shows that if $J(p)$ is connected, then this hypothesis is also necessary.
\begin{theorem}\label{t6}
Let $p$ be a polynomial with connected Julia set and and let $f$ be a Poincar\'e function of~$p$.
If $\area K(p)>0$, then $\area I(f)=0$.
\end{theorem}
Buff and Ch\'eritat~\cite{Buff2012} have shown that there exist polynomials $p$
 with Julia sets of positive measure.
These polynomials $p$ may be chosen to satisfy $J(p)=K(p)$. Theorem~\ref{t6} thus also shows
that the hypothesis that $p$ be semihyperbolic cannot be omitted in Theorem~\ref{t5}.

Theorem~\ref{t6} is a simple consequence of a result of 
Eremenko and Lyubich~\cite[Theorem~7]{Eremenko1992}, using the fact noted above that
$f\in\B$ if $J(p)$ is connected.
This fact also simplifies the proof of Theorem~\ref{t5} considerably if $J(p)$ is connected.
We will thus deal with this special case first and afterwards
provide the additional arguments that have to be made in the general case.

As a second example where our results apply we consider the Weierstra{\ss} $\sigma$-function.
We recall the definition, using the terminology as in~\cite{Ahlfors1953,Hurwitz1964}.
For $\omega_1,\omega_2\in\C\backslash\{0\}$ with $\omega_2/\omega_1\notin\R$
we consider the lattice 
\begin{equation}\label{8a1}
\Omega=\{m\omega_1+n\omega_2\colon m,n\in\Z\}.
\end{equation}
Then
\begin{equation}\label{8a}
\sigma(z)=\sigma(z|\omega_1,\omega_2):=z\prod_{w\in \Omega\backslash\{0\}}
\left(1-\frac{z}{w}\right)\exp\!\left(\frac{z}{w}+\frac12\left(\frac{z}{w}\right)^2\right).
\end{equation}
The Weierstra{\ss} $\zeta$-function and $\wp$-function are defined by
\begin{equation}\label{8b}
\zeta(z)=\frac{\sigma'(z)}{\sigma(z)} 
\quad\text{and}\quad
\wp(z)=-\zeta'(z).
\end{equation}
Moreover, $\eta_1:=2\zeta(\omega_1/2)$.

It can be assumed without loss of generality that $\tau:=\omega_2/\omega_1$ satisfies $\im\tau>0$.
Since $\sigma(c z|c \omega_1,c \omega_2)=c\, \sigma(z|\omega_1,\omega_2)$ for
every $c\in\C\backslash\{0\}$ it suffices to consider the case that $\omega_1=1$
and thus $\tau=\omega_2$.

The Nevanlinna deficiency $\delta(0,\sigma)$ was studied by Gol'dberg~\cite{Goldberg1966}
and Korenkov~\cite{Korenkov1976}; see~\cite{Goldberg2008,Hayman1964} as a
reference for Nevanlinna theory.
The result of the latter paper says that $\delta(0,\sigma)=0$ if and only if
\begin{equation}\label{8c}
\re\!\left(\frac{1}{\eta_1}\right) \geq \frac{\im \tau}{2\pi}.
\end{equation}
We note that the terminology used in~\cite{Goldberg1966,Korenkov1976} is different,
with $\eta_1=\zeta(1/2)$, but
we have converted the result to the terminology of~\cite{Ahlfors1953,Hurwitz1964} introduced above.

The set of all $\tau$ satisfying~\eqref{8c} is shown in Figure~\ref{fig1}.
Since~\cite[p.~8]{Weierstrass1893}
\begin{equation}\label{8d}
\eta_1= 
\pi^2\left(\frac13-2 \sum_{n=1}^\infty \frac{1}{\sin^2 n\pi\tau}\right)
=\frac{\pi^2}{3}\left(1-24 e^{-2\pi\tau}+\O(e^{-4\pi\im\tau})\right)
\end{equation}
as $\im\tau\to \infty$
the upper boundary of this set is very close (but not equal) to
the line given by $\im\tau=6/\pi$.
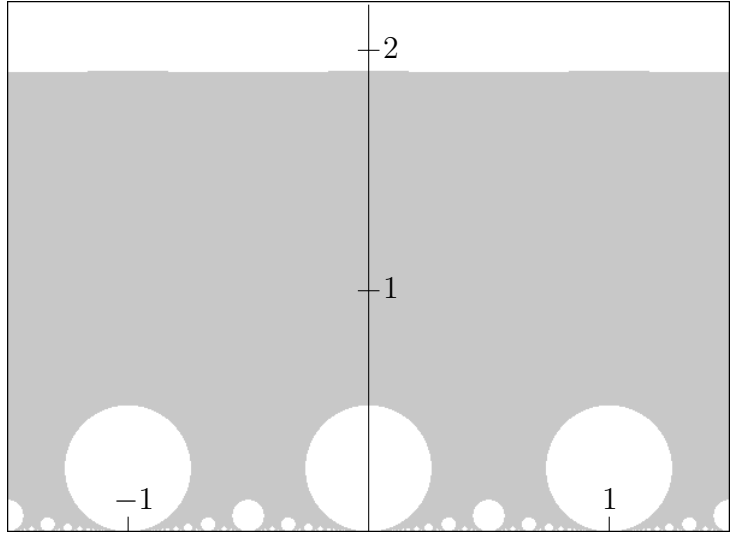
\begin{figure}[htb]
\begin{center}
\begin{overpic}[width=0.7\textwidth]{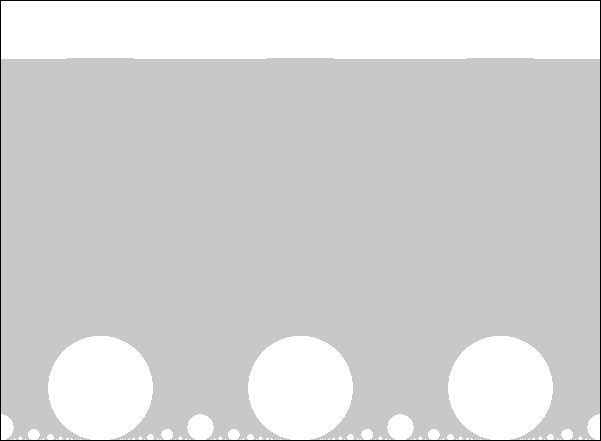}
 \put (50,0) {\line(0,1){73}}
 \put (83.33,0) {\line(0,1){2}}
 \put (82.33,3) {$1$}
 \put (16.67,0) {\line(0,1){2}}
 \put (14.67,3) {$-1$}
 \put (48.5,33.33) {\line(1,0){3}}
 \put (48.5,66.67) {\line(1,0){3}}
 \put (52,32.33) {$1$}
 \put (52,65.67) {$2$}
\end{overpic}
\caption{The set of $\tau$ satisfying \eqref{8c}.}
\label{fig1}
\end{center}
\end{figure}
Consequently, the other boundary components, which are
images of the upper boundary under the modular group, are close to circles.

\begin{theorem}\label{t7}
Suppose that~\eqref{8c} holds.
Then $\area (J(\sigma)\cap A(\sigma))>0$.
\end{theorem}

The results in~\cite{Goldberg1966,Korenkov1976} actually show that if~\eqref{8c} is not satisfied, 
then $\sigma(z)$ tends to $0$ as $z\to\infty$ in some sector. This suggests that $\area I(\sigma)=0$
in this case.

\section{Proofs of Theorems \ref{t1}--\ref{t4}} \label{proofs}
\begin{proof}[Proof of Theorem~\ref{t3}]
Put
\begin{equation}\label{1b}
X=\left\{ z\colon \left|\frac{zf'(z)}{f(z)}\right| \geq n(|z|)^{1/2+\varepsilon}\right\} ,
\quad 
Y=\left\{ z\colon \left|f(z)\right| \geq (1+\varepsilon)|z| \right\},
\end{equation}
and
$W=\Delta\backslash (X\cap Y)$.
Then~\eqref{4b1} takes the form
\begin{equation}\label{1b1}
\logarea W<\infty.
\end{equation}

Note that points in $\C\backslash\{0\}$ which stay in $X\cap Y$ under
iteration of $f$ are contained in $I(f)$.
In order to study the set of such points we consider, for $k\in\N$, the sets
\begin{equation}\label{2a}
A_k =\{z\colon 2^{k}\leq n(|z|)< 2^{k+1} \}.
\end{equation}
For a measurable subset $P$ of $\C$ we then find that
\begin{equation}\label{2b}
\begin{aligned}
 \logarea\!\left(P\cap X\cap A_k\right)
& = 
\int_{P\cap X\cap A_k} \frac{dx\,dy}{|z|^2} 
\\ &
\leq \int_{P\cap X\cap A_k} \frac{ |f'(z)|^2}{n(|z|)^{1+2\varepsilon}|f(z)|^2} dx\,dy
\\ &
\leq 
\frac{1}{2^{(1+2\varepsilon)k}}
\int_{P\cap X\cap A_k} \frac{ |f'(z)|^2}{|f(z)|^2} dx\,dy,
\end{aligned}
\end{equation}
provided the integral on the right hand side exists.
Taking $P=f^{-1}(S)$ for a subset $S$ of $\C$ of finite logarithmic area we find that
\begin{equation}\label{2b1}
\begin{aligned}
 \logarea\!\left(f^{-1}(S)\cap X\cap A_k\right)
& 
\leq \frac{1}{2^{(1+2\varepsilon)k}}
\int_{S} \card\!\left( f^{-1}(w)\cap X\cap A_k\right) \frac{du\,dv}{|w|^2}
\\ &
\leq \frac{2^{k+1}}{2^{(1+2\varepsilon)k}} \logarea S 
= \frac{2}{2^{2\varepsilon k}} \logarea S .
\end{aligned}
\end{equation}
Let $R>1$ be large and 
choose $K\in \N$  such that $2^K\leq n(R)$. We deduce that
\begin{equation}\label{2c}
\begin{aligned}
\logarea\!\left(f^{-1}(S)\cap X\cap \{z\colon |z|\geq R\}\right)
&\leq \sum_{k=K}^\infty 
\logarea\!\left(f^{-1}(S)\cap X\cap A_k\right)
\\ &
\leq 2 \logarea S \sum_{k=K}^\infty \frac{1}{2^{2\varepsilon k}}
\\ &
= \frac{2^{1-2\varepsilon K}}{1- 2^{-2\varepsilon}} \logarea S
\end{aligned}
\end{equation}
and thus
\begin{equation}\label{2d}
\logarea\!\left(f^{-1}(S)\cap X\cap \{z\colon |z|\geq R\}\right)
\leq\frac12 \logarea S 
\end{equation}
if $K$ is sufficiently large, which can be achieved by choosing $R$ large.

Put $S_0=W\cup\{z\colon |z|<R\}$ and $S_k=f^{-1}(S_{k-1})\cap X\cap \{z\colon |z|\geq R\}$
for $k\geq 1$.
Then 
\begin{equation}\label{2d1}
\logarea S_k \leq\frac12 \logarea S_{k-1} \end{equation}
for $k\geq 2$ by~\eqref{2d} and for large $R$ we also have 
\begin{equation}\label{2d2}
\logarea S_1 \leq\frac12 \logarea (S_{0}\cap\Delta).
\end{equation}
Note that 
\begin{equation}\label{2d3}
\begin{aligned}
\logarea (S_{0}\cap\Delta)
&\leq 
\logarea W +\logarea (D(0,R)\cap\Delta)
\\ &
= \logarea W +\log R
<\infty
\end{aligned}
\end{equation}
by~\eqref{1b1}.
It follows from~\eqref{2d1} and~\eqref{2d2} that 
\begin{equation}\label{2f}
\logarea\!\left(\bigcup_{k=0}^\infty S_k\cap\Delta \right)
\leq 2\logarea (S_0\cap\Delta)< \infty.
\end{equation}
Let now
\begin{equation}\label{2g}
T=\left\{z\colon f^k(z)\in X\cap Y \text{ and } \left|f^k(z)\right|\geq R
\text{ for all } k\geq 0 \right\}.
\end{equation}
Here, as usual, $f^0(z)=z$ so that if $z\in T$, then in particular $z\in X\cap Y$ and $|z|\geq R$.

Suppose that $z\in \C\backslash T$. Then there exists $k\geq 0$ such that
$f^k(z)\notin X\cap Y$ or $|f^k(z)|<R$. Thus $f^k(z)\in S_0$.
Assuming $k$ to be minimal we have $f^j(z)\in X\cap Y$ and $|f^j(z)|\geq R$ for $0\leq j\leq k-1$.
We conclude that $f^{k-1}(z)\in S_{1}$.  Inductively we see that $f^{k-j}(z)\in S_j$.
In particular, $z\in S_k$. This implies that 
\begin{equation}\label{2h}
\C\backslash T\subset \bigcup_{k=0}^\infty S_k.
\end{equation}
On the other hand, for $z\in T$ and $k\geq 0$ we have 
\begin{equation}\label{2h1}
|f^k(z)|\geq (1+\varepsilon)^k |z| \geq (1+\varepsilon)^k R,
\end{equation}
by the definition of $T$ and~$Y$.
Hence
$T\subset I(f)$ and thus $\C\backslash I(f)\subset \C\backslash T$.
Together with~\eqref{2f} and~\eqref{2h} this yields~\eqref{4c}.

To prove the second claim we only have to show that if $T\cap F(f)\neq\emptyset$,
then $F(f)$ has a multiply connected component.
Our arguments for this are similar to those in~\cite[Theorem~3.1]{Sixsmith2015}.

So let $z\in T\cap F(f)$. Choose $\delta>0$ such that 
$D(z,\delta)\subset F(f)$. Since $f^n(z)\to\infty$ as $n\to\infty$ we may
assume that $f^n(\zeta)\neq 0$ for all $\zeta\in D(z,\delta)$ and $n\in\N$.
We consider the functions $g_n\colon D(z,\delta)\to\C$, $g_n(\zeta)=f^n(\zeta)/f^n(z)$.
Then $g_n(z)=1$ and, since $f^j(z)\in X$ for $0\leq j\leq n-1$, we have
\begin{equation}\label{2i}
\begin{aligned}
|g_n'(z)|
&=\left|\frac{(f^n)'(z)}{f^n(z)}\right|
=\frac{1}{|f^n(z)|}\prod_{j=0}^{n-1}|f'(f^j(z))|
\\ &
\geq \frac{1}{|f^n(z)|}\prod_{j=0}^{n-1}\frac{n(|f^j(z)|)^{1/2+\varepsilon}|f(f^j(z))|}{|f^j(z)|}
= \frac{1}{|z|}\prod_{j=0}^{n-1}n(|f^j(z)|)^{1/2+\varepsilon}.
\end{aligned}
\end{equation}
Thus $|g_n'(z)|\to\infty$ as $n\to\infty$. Hence the $g_n$ do not form a normal family. 
It now follows easily from Montel's theorem that there exist arbitrarily large $n$
such that $\partial D(0,1)\subset g_n(D(z,\delta))$ or  $\partial D(0,2)\subset g_n(D(z,\delta))$.
In fact, this holds for all large~$n$.
Thus we have $\partial D(0,|f^n(z)|)\subset f^n(D(z,\delta))\subset F(f)$ or 
$\partial D(0,2|f^n(z)|)\subset f^n(D(z,\delta))\subset F(f)$.
 Since $f^n(z)\to\infty$ this implies that $F(f)$ has a multiply connected component.
\end{proof}
\begin{proof}[Proof of Theorem~\ref{t1}]
Upper bounds for $n(r)$ have been given by Hayman and Stewart~\cite[Theorem~5]{Hayman1954},
and we follow the reasoning there. Nevanlinna's first fundamental theorem implies
that there exists a constant $C$ such that 
\begin{equation}\label{3c}
\int_1^r \frac{n(t,a)}{t} dt \leq T(r,f) +C
\end{equation}
for all $a\in\C$ and $r>1$, with the Nevanlinna (or Ahlfors-Shimizu) characteristic $T(r,f)$. Thus
\begin{equation}\label{3d}
n(r,a)=n(r,a) \int_r^{er} \frac{dt}{t} \leq  \int_r^{er} \frac{n(t,a)}{t}dt \leq T(er,f)+C
\leq \log M(er,f)+C
\end{equation}
for all $a\in\C$ and $r>1$.
Given $\delta>0$ we thus have 
\begin{equation}\label{3d1}
n(r)\leq \log M(er,f)+C \leq r^{\rho(f)+\delta}
\end{equation}
for large~$r$.
And for a given $\varepsilon>0$ we may choose $\delta\in (0,\varepsilon]$ such that
\[
n(r)^{1/2+\delta} \leq r^{(\rho(f)+\delta)(1/2+\delta)}\leq r^{\rho(f)/2+\varepsilon}.
\]
We thus deduce from~\eqref{4b} that~\eqref{4b1} holds with $\varepsilon$ replaced by $\delta$.
The conclusion now follows from Theorem~\ref{t3}.
\end{proof}

In order to prove Theorem~\ref{t2} we consider, for $\alpha>0$, the function
\begin{equation}\label{3e}
E_\alpha\colon [0,\infty)\to [0,\infty), \quad E_\alpha(x)=\exp(x^\alpha), 
\end{equation}
and note that there exists $x_\alpha\geq 0$ such that $E_\alpha(x)>x$ for $x>x_\alpha$
and thus $E_\alpha^k(x)\to\infty$ as $k\to\infty$ if $x>x_\alpha$.
We shall use the following lemma which can be deduced from the arguments
in~\cite[Proof of Lemma~3.7]{Cui}, but for completeness we include the proof, following
the reasoning in~\cite{Cui}.
\begin{lemma}\label{la1}
Let $\beta>\alpha>0$. Then there exists $x_0>0$ such that
\begin{equation}\label{3e1}
E_\alpha^{k}(x) \geq E_\beta^{k-2}(x)
\end{equation}
for $k\geq 4$ and $x\geq x_0$.
\end{lemma}
\begin{proof}
Let $F_\alpha(x)=\alpha e^{ x}$. Then
\begin{equation}\label{3e0}
E_\alpha(\exp\exp x)=\exp \exp F_\alpha(x)
\end{equation}
and thus
\begin{equation}\label{3e2}
E_\alpha^k(x)
=\exp \exp F_\alpha^{k}(\log\log x)
=\exp \exp F_\alpha^{k-2}(\alpha x^\alpha)
\end{equation}
for $k\geq 2$. Put $c=\log(2\beta/\alpha)$.  For large $x$ we have 
\begin{equation}\label{3e3}
F_\alpha(x+c)=\alpha e^{x+c} = \alpha e^c e^x= 2 F_\beta(x)\geq F_\beta(x)+c
\end{equation}
and thus
\begin{equation}\label{3e4}
F_\alpha^k(x+c)\geq F_\beta^k(x)+c.
\end{equation}
For large $x$ we also have
\begin{equation}\label{3e5}
F_\alpha(\alpha x^\alpha)\geq x+c
\quad\text{and}\quad
F_\beta(x)\geq \beta x^\beta .
\end{equation}
Combining~\eqref{3e2}, \eqref{3e4} and~\eqref{3e5} we obtain
\begin{equation}\label{3e6}
\begin{aligned}
E_\alpha^{k}(x)
&= \exp \exp(F_\alpha^{k-3}(F_\alpha(\alpha x^\alpha)))
\geq \exp \exp(F_\alpha^{k-3}(x+c))
\\ &
\geq \exp \exp(F_\beta^{k-3}(x)+c)
\geq \exp \exp(F_\beta^{k-3}(x))
\\ &
= \exp \exp(F_\beta^{k-4}(F_\beta(x)))
\geq \exp \exp(F_\beta^{k-4}(\beta x^\beta))
=E_\beta^{k-2}(x)
\end{aligned}
\end{equation}
for $k\geq 4$ and large~$x$.
\end{proof}

\begin{proof}[Proof of Theorem~\ref{t2}]
Let $E_\varepsilon(x)=\exp(x^\varepsilon)$ and, for some large $R>0$, let $B(f)$ be
the set of all $z\in\C$ such that
\begin{equation}\label{3f}
|f^k(z)|\geq E_\varepsilon^k(R)
\end{equation}
for all $k\geq 0$.
We proceed as in the proofs of Theorems~\ref{t3} and~\ref{t1}, with the definition of $Y$ changed
to 
\begin{equation}\label{3f1}
Y=\left\{ z\colon \left|f(z)\right| \geq E_\varepsilon(|z|) \right\},
\end{equation}
however.
Instead of~\eqref{2h1} we now obtain~\eqref{3f} for $z\in T$ and $k\geq 0$.
We deduce 
that~\eqref{4c}, and if $f$ has no multiply connected wandering domains also~\eqref{4d}, hold
with $I(f)$ replaced by $B(f)$.
Thus we only have to show that $B(f)\subset A(f)$.

In order to do so we use the hypothesis that $f$ has finite order. It yields that 
if $\mu>\rho(f)$ and $R$ is sufficiently large, then
$|f(z)|\leq \exp(|z|^\mu)$ for $|z|\geq R$.
With $E_\mu(x)=\exp(x^\mu)$ we thus have
\begin{equation}\label{3g}
M^k(R,f)\leq E_\mu^k(R)
\end{equation}
for all $k\geq 0$.

Applying Lemma~\ref{la1} with $\alpha=\varepsilon$ and $\beta=\mu$ 
we deduce from~\eqref{3f} and~\eqref{3g} that if $z\in B(f)$,
then $|f^k(z)|\geq M^{k-2}(R,f)$ for all $k\geq 4$, provided $R$ has been chosen sufficiently large.
It follows that $z\in A(f)$ and hence $B(f)\subset A(f)$.
\end{proof}
\begin{proof}[Proof of Theorem~\ref{t4}]
Let $0<\delta<\varepsilon$.  It follows from Proposition~\ref{la2} that if 
\begin{equation}\label{5a}
\left|\frac{zf'(z)}{f(z)}\right| < |z|^{\rho(f)/2+\delta},
\end{equation}
then
\begin{equation}\label{5b}
|f(z)|<R\exp\!\left( 4\pi |z|^{\rho(f)/2+\delta}\right)
\leq \exp\!\left( |z|^{\rho(f)/2+\varepsilon}\right),
\end{equation}
if $|z|$ is sufficiently large.
We conclude that if $z$ is in the set occurring on the left hand side of~\eqref{4b2},
with $\varepsilon$ replaced by~$\delta$, and if $|z|$ is sufficiently large,
then $z$ is also in the set occurring on the left hand side of~\eqref{4b3}.
Thus~\eqref{4b2}, with $\varepsilon$ replaced by~$\delta$, follows from~\eqref{4b3}.

It now follows from Theorem~\ref{t2} that the conclusion of Theorem~\ref{t1} holds
with $I(f)$ replaced by $A(f)$.
Moreover, since $f\in\B$, we deduce from the result of Baker~\cite[p.~565]{Baker1984}
 already used after Theorem~\ref{t1} that $F(f)$ has no multiply connected component.
(To rule out multiply connected components of $F(f)$,
we could alternatively use the result of Eremenko and Lyubich~\cite[Theorem~1]{Eremenko1992}
that if $f\in\B$, then $I(f)\subset J(f)$, together with the well-known fact that 
multiply connected components of $F(f)$ are in $I(f)$.)
We thus conclude that~\eqref{4d} holds with $I(f)$ replaced by $A(f)$, as claimed.
\end{proof}

\section{Proofs of Theorems \ref{t5} and \ref{t6}} \label{proofs45}
For the proof of Theorem~\ref{t5} we shall use the following result of Peters and
Smit~\cite[Proposition~10]{Peters2017}.
\begin{lemma}\label{la3}
Let $p$ be a semihyperbolic polynomial. 
Let $A$ be an open set containing all attracting periodic points such 
that $\overline{p(A)}\subset A\subset F(f)$
and let $R>0$ be such that 
$|p(z)|>2R$ for $|z|>R$.
Let $U_0=\{z\colon|z|>R\}\cup A$ and, for $n\in\N$, put $U_n=f^{-n}(U_0)$ and 
$V_n=\C\backslash U_n$.
Then there exist $c_0>0$ and $\theta\in (0,1)$ such that
\begin{equation}\label{5c}
\area V_n \leq c_0\theta^n
\end{equation}
for all $n\in\N$.
\end{lemma}
The proof of Theorem~\ref{t5} is easier if $J(p)$ is connected, because -- as noted
already -- this is equivalent to $f\in\B$
so that Theorem~\ref{t4} can be applied. Therefore we consider this special case first,
and add the arguments required for the general case afterwards.
\begin{proof}[Proof of Theorem~\ref{t5} if $J(p)$ is connected]
Let $R$, $A$, $U_n$ and $V_n$ be as in Lemma~\ref{la3}. Since $p$ does not have 
attracting periodic points we can take $A=\emptyset$. Hence $U_0=\{z\colon|z|>R\}$
and thus
\begin{equation}\label{5d}
V_n=\{z\colon |p^n(z)|\leq R\}.
\end{equation}
Note that if $z\in V_n$, then also $|p^k(z)|\leq R$ for $1\leq k\leq n$. 

We may assume that $R>1$.
Denote by $d$ the degree of~$p$.
It is easy to see that there exists a positive constant $c_1$ such that if $|z|> R$, then
$|p^n(z)|> \exp\!\left(c_1 d^n\right)$.
For example, this follows from B\"ottcher's theorem~\cite[Theorem~9.1]{Milnor2006}
which says that $p$ is conjugate to $z\mapsto z^d$ in some neighborhood of~$\infty$.

Let $\varepsilon>0$. For $n\in\N$ we put 
$m=\lceil \varepsilon n \rceil$ and $W_n=V_m$.
With $c_0$ and $\theta$ as in Lemma~\ref{la3} and $\gamma=\theta^\varepsilon$ we then have
\begin{equation}\label{5e}
\area W_n \leq c_0\theta^m \leq c_0\theta^{\varepsilon n}= c_0 \gamma^n.
\end{equation}
If $z\notin W_n$, then $|p^m(z)|>R$ and thus 
\begin{equation}\label{5f}
\left|p^n(z)\right|=\left|p^{n-m}\left(p^m(z)\right)\right|
\geq \exp\!\left(c_1 d^{n-m}\right)
= \exp\!\left(c_1 d^{\lfloor (1-\varepsilon)n\rfloor}\right).
\end{equation}
With $c_2=c_1/d$ we thus have
\begin{equation}\label{5g}
|p^n(z)|\geq \exp\!\left(c_2 d^{(1-\varepsilon)n}\right) \quad\text{for }z\notin W_n.
\end{equation}

Let now $\lambda\in\C$ with $|\lambda|>1$ be such that
Schr\"oder's functional equation~\eqref{4s} holds.
As noted before Theorem~\ref{t5}, our hypotheses imply that $f\in\B$.

Choosing $r_0\in (0,1]$ sufficiently small we may achieve that $f$ is univalent in $D(0,2r_0)$.
In particular, $f'(z)\neq 0$ for $z\in A:=\{\zeta\colon r_0/|\lambda|\leq |\zeta|\leq r_0\}$.
With
\begin{equation}\label{5h}
S_n:=f^{-1}(W_n)\cap A
\end{equation}
and 
\begin{equation}\label{5h1}
c_3:=\frac{1}{\displaystyle\min_{z\in A}|f'(z)|^2}
\end{equation}
we then have
\begin{equation}\label{5i}
\area S_n\leq c_3 \area W_n.
\end{equation}

For $n\in\N$ we put
\begin{equation}\label{5l}
A_n:=\lambda^n A=\{z\colon |\lambda|^{n-1}r_0\leq |z|\leq |\lambda|^n r_0\}
\quad\text{and}\quad
T_n:=\lambda^n S_n.
\end{equation}
For $|z|\geq r_0$ we now choose $n\in\N$ such that $z\in A_n$.
Then $z$ has the form $z=\lambda^n\zeta$ with $\zeta\in A$.
If $z\in A_n\backslash T_n$, then $\zeta\in A\backslash S_n$ and thus $f(\zeta)\notin W_n$.
Thus~\eqref{5g} yields that
\begin{equation}\label{5j}
|f(z)|=|f(\lambda^n \zeta)|=|p^n(f(\zeta))|\geq \exp\!\left(c_2 d^{(1-\varepsilon)n}\right)
\quad\text{for } z\in A_n\backslash T_n.
\end{equation}
As already mentioned before Theorem~\ref{t5} we have $\rho(f)=\log d/\log|\lambda|$ so that 
$d=|\lambda|^{\rho(f)}$. Noting that $|\lambda|^n\geq |z|/r_0\geq |z|$ for $z\in A_n$ we thus find that
\begin{equation}\label{5k}
d^{(1-\varepsilon)n}=|\lambda|^{(1-\varepsilon)\rho(f)n}
\geq |z|^{(1-\varepsilon)\rho(f)}
\quad\text{for }z\in A_n.
\end{equation}
Combining the last two inequalities we thus find 
that
\begin{equation}\label{5m}
|f(z)|\geq \exp\!\left(c_2 |z|^{(1-\varepsilon)\rho(f)}\right)
\quad\text{for }z\in A_n\backslash T_n.
\end{equation}
Now 
\begin{equation}\label{5n}
\area T_n =|\lambda|^{2n} \area S_n \leq c_0c_3|\lambda|^{2n}\gamma^n
\end{equation}
by~\eqref{5e} and~\eqref{5i} and thus
\begin{equation}\label{5o}
\logarea T_n = \int_{T_n}\frac{dx\,dy}{|z|^2}\leq 
\frac{1}{(|\lambda|^{n-1}r_0)^2}\area T_n
\leq \frac{|\lambda|^2c_0c_3}{r_0^2}\gamma^n.
\end{equation}
We conclude that 
\begin{equation}\label{5p}
T:=\bigcup_{n=1}^\infty T_n
\end{equation}
satisfies
\begin{equation}\label{5q}
\logarea T <\infty.
\end{equation}
On the other hand, we have 
\begin{equation}\label{5r}
\left\{z\colon |z|\geq r_0\text{ and } 
|f(z)|< \exp\!\left(c_2 |z|^{(1-\varepsilon)\rho(f)}\right) \right\}
\subset T
\end{equation}
by~\eqref{5m}.
Thus~\eqref{4b3} holds if $\varepsilon$ is chosen such that
$(1-\varepsilon)\rho(f)>\rho(f)/2+\varepsilon$.

Note that we have not used yet that $J(p)$ is connected.
But since we assume that this is the case, we have $f\in\B$.
Thus~\eqref{4b3} yields the conclusion in view of Theorem~\ref{t4}.
\end{proof}

To deal with the general case, we use the following result of Carleson, Jones
and Yoccoz~\cite[Theorem~2.1]{Carleson1994}, which was also crucial in the proof
of Lemma~\ref{la3} in~\cite{Peters2017}.
Here $\diam A$ denotes the (Euclidean) diameter of a subset $A$ of~$\C$.
\begin{lemma}\label{la4}
Let $p$ be a semihyperbolic polynomial.
Then there exist $\eta>0$, $K_0>0$ and $\tau\in (0,1)$ such that
if $z\in J(f)$, $n\in\N$ and $V$ is a component of $f^{-n}(D(z,\eta))$,
then
\begin{equation}\label{6a}
\diam V \leq K_0\tau^n.
\end{equation}
\end{lemma}
In order to rule out multiply connected wandering domains, we will use the 
following result of Zheng~\cite{Zheng2006}.
\begin{lemma}\label{la4a}
Let $f$ be a transcendental entire function with a multiply connected wandering 
domain~$U$. Then there exist sequences $(r_n)$ and $(R_n)$ satisfying 
$r_n\to\infty$ and $R_n/r_n\to\infty$
such that 
\begin{equation}\label{6a0}
\{z\colon r_n\leq |z|\leq R_n\}
\subset
f^n(U)
\subset
\{z\colon R_{n-1}\leq |z|\leq r_{n+1}\}
\end{equation}
for large~$n$.
\end{lemma}
The conclusion that  $R_n/r_n\to\infty$ was strengthened to $R_n\geq r_n^{1+\varepsilon}$
for some $\varepsilon>0$ in~\cite[Theorem~1.2]{Bergweiler2013}, but we do not need this result here.
\begin{proof}[Proof of Theorem~\ref{t5} in the general case]
We will use the notation and results of the proof given above for
the special case that $J(p)$ is connected.
In particular, the set $T$ defined by~\eqref{5p} satisfies~\eqref{5q} and~\eqref{5r}.
In order to apply Theorem~\ref{t2} it remains to find an upper bound for the size of the set
where $|zf'(z)/f(z)|< |z|^{\rho(f)/2+\varepsilon}$.

To estimate $|f'(z)|$ we note that
\begin{equation}\label{6a1}
\lambda^n f'(\lambda^n \zeta)=(p^n)'(f(\zeta))f'(\zeta)
\end{equation}
by~\eqref{4s}. We are thus looking for an estimate of $\left|\left(p^n)'(z\right)\right|$
for $z\in\C\backslash W_n$. Here, as before, $W_n=V_m$ where $m=\lceil \varepsilon n \rceil$ and
$V_m$ is defined by~\eqref{5d}.
As before we write $p^n(z)=p^{n-k}(p^k(z))$ so that 
\begin{equation}\label{6a2}
(p^n)'(z)=(p^{n-k})'(p^k(z))(p^k)'(z).
\end{equation}
We will then estimate $|(p^k)'(z)|$ for $z\in\C\backslash V_m$, where $k$ is chosen
such that $p^k(z)\in\C\backslash V_0=\{w\colon |w|>R\}$, together with an 
estimate of $|(p^{n-k})'(w)|$ for $|w|>R$.

We may assume that $R$ in~\eqref{5d}
is chosen so large that $|p'(z)|>1$ for $|z|>R$.
In particular, this implies that all critical points of $p$ are contained in $V_0=\overline{D(0,R)}$.
Let $\eta$ be as in Lemma~\ref{la4}.
We may assume that $\eta$ is chosen so small that if $c$ is a critical 
point of $p$ which is not contained in $J(p)$, then $\dist(p^k(c),J(p))>\eta$ for all
$k\geq 0$, where $\dist(\cdot,\cdot)$ denotes the (Euclidean) distance.
This assumption can be made since $p^k(c)\to\infty$ for every critical point $c\notin J(p)$.

There exists $M\in\N$ such that
\begin{equation}\label{6b}
V_{M-1}\subset \left\{ \zeta\colon \dist(\zeta,J(p))\leq \frac12\eta\right\}.
\end{equation}
By the choice of $\eta$ the only critical points of $p$ that are contained in
 $V_{M-1}$ are those that are already contained in $J(p)$.
Together with the choice of $R$ we thus see that the critical points of $p$ that 
are not contained in $J(p)$ are contained in $V_0\backslash V_{M-1}$.

Now $d_0:=\dist(V_{M-1}\backslash V_{M},J(p))$ satisfies $0<d_0<\eta/2$.
We conclude that if $w\in V_{M-1}\backslash V_{M}$, then
$D(w,d_0)\cap J(p)=\emptyset$. 
For $\xi\in J(p)$ we then have $D(w,d_0)\subset D(\xi,\eta)$.

Let now $k>M$ and $z\in V_{k-1}\backslash V_{k}$. 
Then $w=p^{k-M}(z)\in V_{M-1}\backslash V_{M}$.
Denote by $U$ the component of $p^{-(k-M)}(D(w,d_0))$ that contains~$z$.
Then, as just noted, $U$ is contained in a component of $p^{-(k-M)}(D(\xi,\eta))$ 
for some $\xi\in J(p)$ and thus Lemma~\ref{la4} implies that
\begin{equation}\label{6c}
\diam U \leq K_0\tau^{k-M}.
\end{equation}
Since our choice of $\eta$ implies that $D(w,d_0)$ does not intersect the 
orbit of any critical point,
$p^{k-M}\colon U\to D(w,d_0)$ is biholomorphic.
Koebe's one quarter theorem, applied to the inverse $\varphi\colon D(w,d_0)\to U$ of
$p^{k-M}\colon U\to D(w,d_0)$, thus yields that
\begin{equation}\label{6d}
U=\varphi(D(w,d_0))
\supset D\!\left(\varphi(w),\frac14 |\varphi'(w)|d_0\right)
= D\!\left(z,\frac{d_0}{4 |(p^{k-M})'(z)|}\right).
\end{equation}
Hence we can deduce from~\eqref{6c} that if $k\geq M$, then
\begin{equation}\label{6e}
|(p^{k-M})'(z)|\geq \frac{c_4}{\tau^{k-M}}
\quad\text{for }z\in V_{k-1}\backslash V_{k}.
\end{equation}
with $c_4=d_0/(2K_0)$.

Next we note that there exists $c_5,K>0$ such that if $t>0$ and $|z-c|>t$
for every critical point $c$ of $p$, then $|p'(z)|\geq c_5t^{K}$.
It follows that there exists $c_6,L>0$ such that if $1\leq k\leq M$ and $\delta>0$, then
\begin{equation}\label{6f}
\area\!\left\{z\in V_{k-1}\backslash V_{k}\colon \left|(p^{M})'(z)\right|
\leq \delta \right\}\leq c_6 \delta^L .
\end{equation}
In particular, this holds for $k=M$, which together with~\eqref{6e} yields that
\begin{equation}\label{6g}
\area\!\left\{z\in V_{k-1}\backslash V_{k}\colon 
\left|(p^{k})'(z)\right|\leq \frac{c_4\delta}{\tau^{k-M}} \right\}
\leq d^{k-M} \left(\frac{\tau^{k-M}}{c_4}\right)^2  c_6 \delta^L
\end{equation}
for $k>M$.
Since $\tau<1$ we thus have
\begin{equation}\label{6h}
\area\!\left\{z\in V_{k-1}\backslash V_{k}\colon 
\left|(p^{k})'(z)\right|\leq c_4\delta \right\}
\leq c_7 d^{k} \delta^L  
\end{equation}
for $k>M$, with $c_7=c_6/c_4^2$.
We may assume that $c_4\leq 1$ so that~\eqref{6h} also holds for $1\leq k\leq M$ by~\eqref{6f}.

Next, as explained after~\eqref{6a2}, we 
want to
estimate $|(p^j)'(w)|$ for $|w|> R$.
In order to do so, let $g$ be the Green function of the (super)attracting basin of~$\infty$. Then 
\begin{equation}\label{6j}
g(p(z))=dg(z).
\end{equation}
This implies that $g(p^j(z))=d^j g(p(z))$ and thus 
\begin{equation}\label{6k}
|\nabla g(p^j(z))|\cdot | (p^j)'(z)|=d^j|\nabla g(z)|.
\end{equation}
We have $g(z)=\log|z|+c+o(1)$ as $z\to\infty$ for some constant $c$. It is not difficult
to show that this implies that 
\begin{equation}\label{6l}
|\nabla g(z)|\sim \frac{1}{|z|}
\end{equation}
as $z\to\infty$. Hence 
\begin{equation}\label{6m}
\left| \frac{(p^j)'(z)}{p^j(z)}\right|\sim d^j|\nabla g(z)|
\end{equation}
as $j\to\infty$.
Using~\eqref{6l} again we deduce that there exists a positive constant $c_8$ such that 
\begin{equation}\label{6n}
\left| \frac{(p^j)'(w)}{p^j(w)}\right|\geq \frac{c_8}{|w|} d^j
\quad\text{for } |w|\geq R.
\end{equation}

Recall from the proof for the special case that $J(p)$ is connected
that for $n\in\N$ we put $m=\lceil \varepsilon n \rceil$ and $W_n=V_m$.
With $\alpha= d^{-2\varepsilon/L}$ we now put
\begin{equation}\label{6o}
W_n'=
\bigcup_{k=1}^m\left\{z\in V_{k-1}\backslash V_{k}\colon \left|(p^{k})'(z)\right|\leq c_4\alpha^n \right\}.
\end{equation}
and deduce from~\eqref{6h} that
\begin{equation}\label{6p}
\area W_n'\leq c_7 \alpha^{Ln} \sum_{k=1}^m d^{k}  
\leq \frac{c_7d}{d-1} \alpha^{Ln} d^{m}  
\leq c_9 \alpha^{Ln} d^{\varepsilon n}  
= c_9 d^{-\varepsilon n}  
\end{equation}
with $c_9=c_7d^2/(d-1)$.

Let now $z\in \overline{D(0,R)}\backslash (W_n\cup W_n')=V_0\backslash (V_m\cup W_n')$. Then
$z\in V_{k-1}\backslash V_k$
for some $k\in \{1,\dots,m\}$.
Since $z\notin W_n'$ we have $\left|(p^{k})'(z)\right|> c_4\alpha^n$.
Moreover, $w:=p^k(z)$ satisfies 
$R<|w|\leq M(R,p)$. Together with~\eqref{6n} we thus find  with $c_{10}=c_4c_8/(M(R,p)d)$ that
\begin{equation}\label{6q}
\begin{aligned}
\left| \frac{(p^n)'(z)}{p^n(z)}\right|
&=
\left| \frac{(p^{n-k})'(w)}{p^{n-k}(w)}(p^k)'(z)\right|
 >  \frac{c_8}{|w|}  d^{n-k}c_4\alpha^n
\\ &
\geq \frac{c_4c_8}{M(R,p)} d^{n-m}\alpha^n
\geq c_{10} d^{(1-\varepsilon)n}\alpha^n
= c_{10} d^{(1-\varepsilon-2\varepsilon/L)n}.
\end{aligned}
\end{equation}
With
$\varepsilon'=\varepsilon+2\varepsilon/L$ we thus have
\begin{equation}\label{6r}
\left| \frac{(p^n)'(z)}{p^n(z)}\right| \geq  c_{10}  d^{(1-\varepsilon')n}
\quad\text{for } z\in \overline{D(0,R)}\backslash (W_n\cup W_n').
\end{equation}
Similarly as in~\eqref{5h} we consider
\begin{equation}\label{6s}
S_n':=f^{-1}(W_n\cup W_n')\cap A
\end{equation}
and deduce, analogously to~\eqref{5i}, that
\begin{equation}\label{6t}
\area S_n'\leq c_3 ( \area W_n+ \area W_n' ).
\end{equation}
In analogy to the previous arguments we put $T_n'=\lambda^n S_n'$.
Writing $z\in A_n$ in the form $z=\lambda^n\zeta$ with $\zeta\in A$
we have
\begin{equation}\label{6u}
\frac{zf'(z)}{f(z)} =
\frac{\lambda^n \zeta f'(\lambda^n \zeta )}{f(\lambda^n \zeta )} =
\frac{(p^n)'(f(\zeta)) \zeta f'(\zeta)}{p^n(f(\zeta))} 
\end{equation}
by~\eqref{6a1}.
Using~\eqref{5h1} and~\eqref{6r} we deduce that
\begin{equation}\label{6v}
\left|\frac{zf'(z)}{f(z)} \right| \geq  c_{11}  d^{(1-\varepsilon')n}
\quad\text{for } z\in A_n\backslash T_n'
\end{equation}
with $c_{11}=c_{10} r_0/(|\lambda|\sqrt{c_3})$.
It thus follows from~\eqref{5k} that 
\begin{equation}\label{6w}
\left|\frac{zf'(z)}{f(z)} \right| \geq  c_{11}  |z|^{(1-\varepsilon')\rho(f)}
\quad\text{for } z\in A_n\backslash T_n'.
\end{equation}
In analogy to~\eqref{5p}, \eqref{5q} and~\eqref{5r} we now deduce from~\eqref{6t},
\eqref{6p} and~\eqref{5e} that the set $T'$ defined by
\begin{equation}\label{5p1}
T':=\bigcup_{n=1}^\infty T_n'
\end{equation}
satisfies
\begin{equation}\label{5q1}
\logarea T' <\infty
\end{equation}
and
\begin{equation}\label{5r1}
\left\{z\colon |z|\geq r_0\text{ and }
\left|\frac{zf'(z)}{f(z)} \right| <  c_{10}  |z|^{(1-\varepsilon')\rho(f)}
\right\}
\subset T'.
\end{equation}
It follows from~\eqref{5q1} and~\eqref{5r1}, together with~\eqref{5q} and~\eqref{5r},
that~\eqref{4b2} holds if $\varepsilon$ and hence $\varepsilon'$ are sufficiently small.
The conclusion will thus follow from Theorem~\ref{t2} if we can show that
$f$ does not have multiply connected wandering domains.

In order to do so, let $u_0\in\C$ such that $v_0:=f(u_0)\in J(p)$. It follows from~\eqref{4s} that
\[
f(\lambda^n u_0)=p^n(f(u_0))=p^n(v_0)\in J(p)
\]
and thus $|f(\lambda^n u_0)|\leq R$ for all $n\in\N$.
Lemma~\ref{la4a} now implies that $f$ does not have multiply connected wandering domains.
\end{proof}
The result of Eremenko and Lyubich~\cite[Theorem~7]{Eremenko1992} already mentioned in the
introduction that we will use is the following.
\begin{lemma}\label{la5}
Let $f\in\B$ and suppose that there exists $R>0$ such that
\begin{equation}\label{7a}
\liminf_{r\to\infty} \frac{\logarea\!\left( f^{-1}(D(0,R))\cap D(0,r)\cap \Delta \right)}{\log r}>0.
\end{equation}
Then $\area I(f)=0$.
\end{lemma}
\begin{proof}[Proof of Theorem~\ref{t6}]
Assume that~\eqref{4s} holds and that $\area K(p)>0$.
Put $L=f^{-1}(K(p))$ and choose $R>0$ such that $K(p)\subset D(0,R)$.
It follows that $L\subset f^{-1}(D(0,R))$ and $\area L>0$.
Since $K(p)$ is invariant under $p$, we can deduce from~\eqref{4s}
that $L$ is invariant under the map $z\mapsto\lambda z$.
Thus also $A:=\area\!\left( L \cap \{z\colon 1\leq |z|\leq |\lambda\}\right)>0$.
For $r>1$ we choose $n\in\N$ with $|\lambda|^{n-1}\leq r<|\lambda|^n$.
Hence 
\begin{equation}\label{7b}
\begin{aligned}
\logarea \!\left( L \cap D(0,r)\cap \Delta\right)
&\geq
\logarea \!\left( L \cap D(0,|\lambda|^{n-1})\cap\Delta \right)
\\ &=
(n-1) \logarea \!\left( L \cap D(0,|\lambda|)\cap\Delta \right)
\\ &
\geq  (n-1) \frac{A}{|\lambda|^2}
\geq   \frac{n-1}{n} \frac{A}{|\lambda|^2\log|\lambda|}\log r .
\end{aligned}
\end{equation}
Since $L\subset f^{-1}(D(0,R))$ and since $n$ tends to $\infty$ with $r$
we deduce that the lower limit on the 
left hand side of~\eqref{7a} is at least $A/(|\lambda|^2\log|\lambda|)$.
The conclusion now follows from Lemma~\ref{la5}.
\end{proof}
\section{Proof of Theorem \ref{t7}} \label{proof7}
It is well-known that $\rho(\sigma)=2$. This is also an immediate consequence of 
the following lemma, which is a special case of the asymptotics of $\sigma$ and $\zeta$ that were
obtained in~\cite{ZajacKorenkov2015}.
Here we put $w_{mn}=m\omega_1+n\omega_2$ for $m,n\in\Z$.
\begin{lemma}\label{la6}
Let  
\begin{equation}\label{9a}
E=\bigcup_{m,n\in\Z} D\!\left(w_{m,n},e^{-|w_{n,m}|}\right)
\end{equation}
and
\begin{equation}\label{9b}
F=\bigcup_{m,n\in\Z} D\!\left(w_{m,n},\frac{1}{\sqrt{|w_{n,m}|}}\right).
\end{equation}
Then 
\begin{equation}\label{9c}
\log|\sigma(z)|= V(z)
+\O(|z|)
\quad\text{as }|z|\to\infty, \;  z\notin E,
\end{equation}
where 
\begin{equation}\label{9c1}
V(z)=\frac{\pi}{2\im\tau} |z|^2 
+\re\!\left(\!\left( \frac{\eta_1}{2}-\frac{\pi}{2\im\tau}\right)\!z^2\right),
\end{equation}
and
\begin{equation}\label{9d}
\zeta(z)= \eta_1 z- \frac{2\pi i}{\im\tau} \im z
+\O\!\left(\sqrt{|z|}\right)
\quad\text{as }|z|\to\infty, \;  z\notin F.
\end{equation}
\end{lemma}
\begin{proof}[Proof of Theorem \ref{t7}]
First we note that the condition~\eqref{8a} is equivalent to 
\begin{equation}\label{9e}
\left|\eta_1-\frac{\pi}{\im\tau}\right|\leq  \frac{\pi}{\im\tau}.
\end{equation}
This means that the second term on the right hand side of~\eqref{9c1} is not bigger than the 
first term.

Put $B=\pi/\im\tau$. The last inequality says that there exist $A\in [0,B]$ and
$\alpha\in (-\pi,\pi]$ such that $\eta_1-B=Ae^{i\alpha}$.
With these abbreviations~\eqref{9c1} takes the form
\begin{equation}\label{9e1}
V(z)=\frac{1}{2} ( B|z|^2 +\re(Ae^{i\alpha}z^2))  
\end{equation}
which we may also write as
\begin{equation}\label{9c2}
V(re^{i\theta} )=\frac{1}{2} ( B +A\cos(\alpha+2\theta))  r^2 .
\end{equation}
Put $\theta^\pm=(\pm\pi-\alpha)/2$ and
\begin{equation}\label{9f}
G=
\{re^{i\theta}\colon |\theta-\theta^+|\leq r^{-1/4}\text{ or }|\theta-\theta^-|\leq r^{-1/4}\} .
\end{equation}
Then
\begin{equation}\label{9g}
\begin{aligned}
V(re^{i\theta} )
&\geq \frac{1}{2} ( B +A\cos(\pi+r^{-1/4}))  r^2 
\geq \frac{B}{2}(1- \cos(r^{-1/4}))  r^2 
\\ &
=(1+o(1) \frac{B}{4}  r^{3/2}
\quad\text{as }r\to\infty,\; re^{i\theta}\notin G.
\end{aligned}
\end{equation}
It thus follows from~\eqref{9c} that there exists a positive constant $c_1$ such that
\begin{equation}\label{9h}
\log|\sigma(z)| \geq c_1  |z|^{3/2}
\quad\text{for } z\in \Delta\backslash (E\cup G).
\end{equation}

To estimate $z\sigma'(z)/\sigma(z)=z\zeta'(z)$ we note that
\begin{equation}\label{9i}
\eta_1 z - \frac{2\pi i}{\im\tau} \im z
= (B+Ae^{i\alpha}) z - 2 i B \im z
= B  \overline{z} + Ae^{i\alpha} z
\end{equation}
and hence
\begin{equation}\label{9j}
z\left( \eta_1 z - \frac{2\pi i}{\im\tau} \im z\right)= B|z|^2 +Ae^{i\alpha} z^2.
\end{equation}
Combining this with~\eqref{9e1} we see that
\begin{equation}\label{9k}
\begin{aligned}
\left|z\!\left( \eta_1 z - \frac{2\pi i}{\im\tau} \im z\right)\right|
&\geq \re\!\left(z\!\left( \eta_1 z - \frac{2\pi i}{\im\tau} \im z\right)\!\right)
\\ &
= B|z|^2 +A\re\!\left(e^{i\alpha} z^2\right)
=2V(z).
\end{aligned}
\end{equation}
Together with~\eqref{9d} and~\eqref{9g} this implies that
there exists a constant $c_2$ such that
\begin{equation}\label{9l}
\left|\frac{z\sigma'(z)}{\sigma(z)}\right| =|z\zeta(z)| \geq c_2  |z|^{3/2}
\quad\text{for } z\in \Delta\backslash (F\cup G).
\end{equation}
It is easy to see that $\logarea(\Delta\cap (E\cup F\cup G))<\infty$.
Hence~\eqref{9h} and~\eqref{9l} say that~\eqref{4b2} holds for $f=\sigma$ if $0<\varepsilon<\frac12$.
Since Lemma~\ref{la4a} implies that $f$ has no multiply connected wandering domains, 
the conclusion now follows from Theorem~\ref{t2}.
\end{proof}
\section{Remarks} \label{remarks}
\begin{remark} \label{rem1}
The main tool used by Eremenko and Lyubich~\cite{Eremenko1992}
in their proof of Proposition~\ref{la2} is a logarithmic change of
variable which consists of considering the function $F(\zeta)=\log f(e^\zeta)$
in certain domains.
With $z=e^\zeta$ we have $F'(\zeta)=zf'(z)/f(z)$. 
In our results we also use the expression $zf'(z)/f(z)$, even though we do not assume 
that $f\in\B$ anymore.

We mention that the quantity $zf'(z)/f(z)$ also appears in~\cite{Bergweiler2016}
and in~\cite{Sixsmith2015}. The result in~\cite[Theorem~1.4]{Bergweiler2016}
required lower bounds for $\re(zf'(z)/f(z))$ while our results only assume bounds for $|zf'(z)/f(z)|$.
We note, however, that~\eqref{9k} also yields lower bounds for $\re(z\sigma'(z)/\sigma(z))$.
\end{remark}
\begin{remark} \label{rem2}
Besides the Lebesgue measure of $J(\sin(\alpha z+\beta))$, 
McMullen~\cite[Theorem 1.2]{McMullen1987} also considered the Hausdorff dimension of $J(\lambda e^z)$.
This result and the techniques used in its proof have been the starting point of many
results on the Hausdorff dimension of Julia sets; see~\cite{Stallard2008} for a survey 
and, e.g., \cite{Baranski2009,Bergweiler2010,Bergweiler2009,Rempe2010,Sixsmith2015b} for some more
recent results.

The methods in~\cite{Bergweiler2010,Sixsmith2015b} also use estimates of $zf'(z)/f(z)$, 
but otherwise they are quite different from the ones employed here.
\end{remark}
\subsection*{Acknowledgements}
I thank Weiwei Cui and the referee for helpful comments.

%
%


\begin{thebibliography}{99}
\bibitem{Ahlfors1953} Lars V. Ahlfors, Complex analysis. An introduction to the theory of
analytic functions of one complex variable. McGraw-Hill Book Company, Inc.,
New York-Toronto-London, 1953.
\bibitem{Aspenberg2012}
Magnus Aspenberg and Walter Bergweiler,
Entire functions with Julia sets of positive measure.
Math. Ann. 352 (2012), 
no.~1,
27--54.
\bibitem{Baker1984} I. N. Baker, 
Wandering domains in the iteration of entire functions,
{\rm Proc. London Math. Soc.} (3) 49 (1984), 
no.~3,
563--576.
\bibitem{Baranski2009}
Krzysztof Bara\'nski, Bogus\l awa Karpi\'nska and Anna Zdunik,
Hyperbolic dimension of Julia sets of meromorphic maps with logarithmic tracts.
Int. Math. Res. Not. IMRN 2009, 
no.~4,
615--624.
\bibitem{Bergweiler1993}
Walter Bergweiler,
Iteration of meromorphic functions.
{\rm Bull.\ Amer.\ Math.\ Soc.\ (N.\ S.)}
29 (1993), no.~2,
151--188.
\bibitem{Bergweiler1995}
Walter Bergweiler,
On the zeros of certain homogeneous differential polynomials.
Arch. Math. (Basel) 64 (1995),
no.~3,
199--202.
\bibitem{Bergweiler2016}
Walter Bergweiler and Igor Chyzhykov,
Lebesgue measure of escaping sets of entire functions of completely regular growth.
{\rm J. London Math. Soc.} 94 (2016), 
no.~2,
639--661.
\bibitem{Bergweiler1999}
Walter Bergweiler and A.\ Hinkkanen,
On semiconjugation of entire functions.
Math. Proc. Cambridge Philos. Soc. 126 (1999), 
no.~3,
565--574.
\bibitem{Bergweiler2010}
Walter Bergweiler and Bogus\l awa Karpi\'nska,
On the Hausdorff dimension of the Julia set of a regularly growing
entire function.
Math. Proc. Cambridge Philos. Soc. 148 (2010),
no.~3,
531--551.
\bibitem{Bergweiler2009}
Walter Bergweiler, Bogus\l awa Karpi\'nska and Gwyneth M.\  Stallard, 
The growth rate of an entire function and the Hausdorff dimension of its Julia set.
J. London Math. Soc. 80 (2009), 
no.~3,
680--698.
\bibitem{Bergweiler2013}
Walter Bergweiler, Philip J. Rippon and Gwyneth M. Stallard,
Multiply connected wandering domains of entire functions.
Proc. London Math. Soc. (3) 107 (2013), no.~6, 1261--1301.
\bibitem{Buff2012}
Xavier Buff and Arnaud Ch\'eritat,
Quadratic Julia sets with positive area.
Ann. of Math. (2) 
176 (2012), 
no.~2,
673--746.
\bibitem{Bishop2015}
Christopher J. Bishop,
Constructing entire functions by quasiconformal folding.
Acta Math. 214 (2015), 
no.~1, 
1--60.
\bibitem{Carleson1994} Lennart Carleson, Peter W. Jones and Jean-Christophe Yoccoz,
Julia and John. Bol. Soc. Bras. Mat. 25 (1994), 
no.~1,
1--30.
\bibitem{Cui}
Weiwei Cui,
Lebesgue measure of escaping sets of transcendental entire functions in the Eremenko-Lyubich class.
Preprint, arXiv: 1608.04600.
\bibitem{Epstein2015}
Adam Epstein and Lasse Rempe-Gillen,
On invariance of order and the area property for finite-type entire functions.
Ann. Acad. Sci. Fenn. Math. 40 (2015), 
no.~2,
573--599.
\bibitem{Eremenko1989}
A.\ E.\ Eremenko, On the iteration of entire functions. In {\rm
``Dynamical systems and ergodic theory''}. Banach Center Publications
23, Polish Scientific Publishers, Warsaw 1989, pp.\ 339--345.
\bibitem{Eremenko1992}
A.\ E.\ Eremenko and M.\ Yu.\ Lyubich, Dynamical properties of some
classes of entire functions. {\rm Ann.\ Inst.\ Fourier} 42 (1992),
no.~4,
989--1020.
\bibitem{Goldberg1966}
A.\ A.\ Gol'dberg 
Distribution of values of the {W}eierstrass sigma-function.
Izv. Vys\v s. U\v cebn. Zaved. Matematika 1966, no.~1 (50), 43--46.
\bibitem{Goldberg2008}
Anatoly A.\ Goldberg and Iossif V.\ Ostrovskii,
Value distribution of meromorphic functions.
Transl.\ Math.\ Monographs 236, American Math.\ Soc., Providence, R.~I., 2008.
\bibitem{Hayman1964}
W.\ K.\ Hayman,
{\rm Meromorphic Functions},
Clarendon Press, Oxford, 1964.
\bibitem{Hayman1954}
W. K. Hayman and F. M. Stewart,
Real inequalities with applications to function theory.
Math. Proc. Cambridge Philos. Soc. 50 (1954), no.~2,
250--260.
\bibitem{Hurwitz1964} Adolf Hurwitz,
Vorlesungen \"uber allgemeine Funktionentheorie und elliptische
Funktionen. Springer-Verlag, Berlin-G\"ottingen-Heidelberg-New York, 1964.
\bibitem{Korenkov1976}
N.\ E.\ Korenkov,
The distribution of the values of the Weierstrass sigma function (Russian).
Mathematics collection (Russian),
Izdat. ``Naukova Dumka'', Kiev, 1976,
pp.\ 240--242.
\bibitem{Mihaljevic2012}
Helena Mihaljevi\'c-Brandt and J\"orn Peter,
Poincar\'e functions with spiders' webs.
Proc. Amer. Math. Soc. 140 (2012), 
no.~9, 
3193--3205. 
\bibitem{McMullen1987}
Curt McMullen, Area and Hausdorff dimension of Julia sets of entire
functions. Trans.\ Amer.\ Math.\ Soc.\ 300 (1987), 
no.~1,
329--342.
\bibitem{Milnor2006}
John Milnor,
Dynamics in one complex variable.
Third edition. Annals of Mathematics Studies, 160. Princeton University Press, Princeton, NJ, 2006.
\bibitem{Peters2017}
Han Peters and Iris Marjan Smit,
Fatou components of attracting skew products. 
J. Geom. Anal. (2017); doi:10.1007/s12220-017-9811-6.
\bibitem{Rempe2010}
Lasse Rempe and Gwyneth M.\ Stallard,
Hausdorff dimensions of escaping sets of transcendental entire functions.
Proc. Amer. Math. Soc. 138 (2010),
no.~5,
1657--1665.
\bibitem{Rippon2005a}
P.\ J.\  Rippon and G.\ M.\  Stallard,
On questions of Fatou and Eremenko.
Proc. Amer. Math. Soc.  133  (2005), 
no.~4,
1119--1126.
\bibitem{Rippon2012}
P.\ J.\  Rippon and G.\ M.\  Stallard,
Fast escaping points of entire functions. 
Proc. Lond. Math. Soc. (3) 105 (2012), 
no.~4,
787--820.
\bibitem{Schleicher2010} 
Dierk Schleicher,
Dynamics of entire functions.
In ``Holomorphic dynamical systems''.
Lecture Notes Math. 1998, Springer, Berlin, 2010,  pp.~295--339,
\bibitem{Sixsmith2015b}
D. J. Sixsmith,
Functions of genus zero for which the fast escaping set has Hausdorff dimension two.
Proc. Amer. Math. Soc. 143 (2015), 
no.~6,
2597--2612. 
\bibitem{Sixsmith2015}
David J. Sixsmith,
Julia and escaping set spiders' webs of positive area.
Int. Math. Res. Not. IMRN
2015, 
no.~19,
9751--9774.
\bibitem{Stallard2008}
G.\ M.\ Stallard,
Dimensions of Julia sets of transcendental meromorphic functions.
In ``Transcendental Dynamics and Complex Analysis''.
London Math.\ Soc.\ Lect.\ Note Ser.\ 348.
Edited by P.\ J.\ Rippon
and G.\ M.\ Stallard,
Cambridge Univ.\ Press, Cambridge, 2008, pp.~425--446.
\bibitem{Valiron1923}
Georges Valiron,
Lectures on the general theory of integral functions.
{\'Edouard Privat}, {Toulouse}, 1923.
\bibitem{Weierstrass1893} K. Weierstrass,
Formeln und Lehrs\"atze zum Gebrauche der elliptischen Functionen.
Edited by H. A. Schwarz, Springer, Berlin, 1893.
\bibitem{ZajacKorenkov2015} 
J. Zaj\k{a}c, M. E. Korenkov and Yu. I. Kharkevych,
On the asymptotics of some Weierstrass functions.
Ukrainian Math. J. 67 (2015),  no.~1, 
154--158.
\bibitem{Zheng2006} Jian-Hua Zheng,
On multiply-connected Fatou components in iteration of meromorphic functions.
J.~Math. Anal. Appl.
313 (2006), no.~1, 24--37.
\end{thebibliography}
\end{document}